\documentclass[11pt]{article} 

\usepackage{graphicx}
\usepackage{amssymb}
\usepackage{amsmath}
\usepackage{amsthm}
\usepackage{amsfonts}
\usepackage{bm}
\usepackage{color}

\newtheorem{theorem}{Theorem}

\newtheorem{remark}{Remark}
\newtheorem{proposition}[theorem]{Proposition}

\usepackage[hidelinks]{hyperref}

\newcommand{\R}{\mathbb{R}}

\newcommand{\Pmk}{\mathcal{P}^-_{k}}
\newcommand{\Pmo}{\mathcal{P}^-_{1}}

\title{\textbf{Allen-Cahn equation for the truncated Laplacian: unusual phenomena}}
\author{I. Birindelli, G. Galise\\
Dipartimento di Matematica, Sapienza Universit\`a di Roma}
\date{}

\begin{document}
\maketitle
\begin{abstract}
We study entire viscosity solutions of the Allen-Cahn type equation for the truncated Laplacian that are either one dimensional or radial, in order to shed some light on a possible extension of the Gibbons conjecture in this degenerate elliptic setting.
\end{abstract}

\bigskip
{\small
\textbf{Acknowledgements:} This research is partially supported by INDAM-GNAMPA.}

\section{Introduction}
In this paper we consider some entire solutions of a degenerate elliptic equation
with non linear forcing term, where the behaviour of the solutions is quite different from 
the uniformly elliptic case. 

We begin by recalling the linear uniformly elliptic model problem we have in mind in order to underline the 
differences with the degenerate elliptic case that we shall treat here.
Consider the entire solutions of the Allen-Cahn equation
\begin{equation}\label{dgib} \Delta u +u-u^3=0\quad  \mbox{in }\ \R^N\end{equation}
where $N \geq 2$.

Writing the variable as $x = (x',x_N)$, with $x_N \in \R$, we are interested in solutions that satisfy the following conditions:
\begin{equation}\label{Gib}
 |u| <1 \ \mbox{ and}\ \lim_{x_N\rightarrow \pm\infty} u(x',x_N) = \pm 1\ \mbox{ uniformly in}\ x' \in \R^{N-1}.
\end{equation}
Any solution of \eqref{dgib} satisfying \eqref{Gib} is necessarily a function of $x_N$, i.e. it is independent of $x'$. This was referred to as Gibbons conjecture and it has been proved by several authors
\cite{ BBG, BHM, F} and then it has been generalized to different contexts \cite{BM, P}.
It is well known that this conjecture is somehow related to a conjecture of De Giorgi, see \cite{AC}, but it would be too long to dwell on the subject.

The function $f(u)=u-u^3$ can be replaced by a much more general class of functions and the result is still valid. 
On the other hand, if instead of the Laplacian one consider some degenerate elliptic operators there are a few results, let us mention e.g. the works in the Heisenberg group and in Carnot groups \cite{BL,BJ}.

In this work we shall mainly focus on equations whose leading term is the operator $\Pmk$ which is sometimes referred to as the truncated Laplacian. The operator $\Pmk$ is defined, for any  $N\times N$ symmetric matrix $X$, by the partial sum
$$
\Pmk(X)=\lambda_1(X)+\ldots+\lambda_k(X)
$$ 
of the ordered eigenvalues $\lambda_1\leq\ldots\leq\lambda_N$ of $X$.
We shall consider solutions of the equation
\begin{equation}\label{eq0}
\Pmk(D^2u)+f(u)=0\quad\text{in \;$\mathbb R^N$},
\end{equation}
for a general class of functions $f$ modelled on $f(u)=u-u^3$.
Clearly $\Pmk(D^2u)$ corresponds to the Laplacian when $k=N$, hence in the whole 
paper we shall suppose that $k=1, \dots, N-1$.

These operators are degenerate elliptic in the sense that 
$$X\leq Y\Rightarrow \Pmk(X)\leq \Pmk(Y).$$
In order to emphasize the strong degeneracy of these operators, let us mention that 
for any matrix $X$ there exists $M\geq 0$, not identically zero, such that 
$$\Pmk(X)=\Pmk(X+M).$$ 
It is immediate to see that, for example, one can take 
$M=v_N\otimes v_N$ where $v_N$ is an eigenvector corresponding to the largest eigenvalue of $X$. 

In previous works we have, together with Hitoshi Ishii and/or Fabiana Leoni,  
encountered a certain number of surprising results, related to  these degenerate 
operators.  
It would be too long to recall them all here but we shall 
briefly recall some of those closer to the results in this paper.

The classical Liouville result states that any harmonic function which is bounded from below is a constant. This is not true for bounded from below entire solutions of $\Pmk(D^2u)=0$. 

Similarly, concerning the semi-linear Liouville theorem, the existence of entire non negative solutions  of 
$F(D^2u)+u^p=0$ is quite different if $F(D^2u)=\Delta u$ or if $F(D^2u)=\Pmk(D^2u)$.
Indeed, for the Laplacian there is a threshold for $p$ between existence and non existence. This threshold is different for solutions and supersolutions and it depends on the dimension of the space. Instead, 
for solutions of the equation 
\begin{equation}\label{3primo}
\Pmk(D^2u)+u^p=0\quad \mbox{ in }\ \R^N
\end{equation}
the following hold:
\begin{enumerate}
\item For any $p>0$ there exist nonnegative viscosity solutions $u\not\equiv0$;
\item For any $p\geq1$ there exist positive classical solutions;
\item For $p<1$ there are no positive viscosity supersolutions of \eqref{3primo}.
\end{enumerate}
Interestingly, if one considers  instead non positive solutions, the non existence results 
are very similar to those of the Laplacian in $\R^k$. These Liouville  theorems were proved in \cite{BGL}.

Hence this lead us to wonder what happens to bounded entire solutions of the equation
\begin{equation}\label{01}\Pmk(D^2u)+u-u^3=0\end{equation} 
which satisfy $|u|<1$ and which,  a priori, may change sign.
Can one expect solutions to be one dimensional? 
Precisely, can we extend Gibbons conjecture to this degenerate case?

In order to answer these questions the first step is to study the one dimensional 
solutions, which is what we do in Section \ref{oned}. Interestingly the results are 
completely different from the uniformly elliptic case and this leads to different 
conjectures. 
Precisely, we consider only one dimensional solutions $u(x',x_N)=v(x_N)$ 
of \eqref{01} that satisfy $|u|<1$.
The results can be summarized in the following way:
{\em
\begin{enumerate}
\item The only classical solution is $u\equiv 0$; 
\item Any viscosity subsolution is non negative and there exists a  non trivial viscosity solution that satisfies
\begin{equation}\label{Gib2}
  \lim_{x_N\rightarrow -\infty} v(x_N) =0\  \mbox{ and } \lim_{x_N\rightarrow \infty} v(x_N) =1.
\end{equation}
\item There are no solutions that are strictly monotone or positive
\end{enumerate}
}
Observe that, even though $u^3-u$ is the derivative of the double well potential $F(u)=\frac14(1-u^2)^2$, the lack of ellipticity does not allow the solutions to go from $-1$ to $1$. This is the first surprising result.

\medskip
Hence Gibbons conjecture should be reformulated in the following way:

\smallskip
\noindent{\bf Question 1.} {\em
Is it true that if $u$ is a solution of \eqref{01}, $|u|<1$, satisfying 
$$ \lim_{x_N\rightarrow -\infty} u(x',x_N) =-1\  \mbox{ and } \lim_{x_N\rightarrow \infty} u(x',x_N) =1\ \mbox{uniformly in}\ x'\in\R^{N-1}$$ 
then $u$ is 1 dimensional?}
\\
If the answer was positive, this would imply that there are no solutions of \eqref{01} that satisfies \eqref{Gib}, since such one dimensional solutions don't exist in view of point 2 above. Or, equivalently, if such a solution of \eqref{01}  exists then the answer to Question 1 is negative.

\smallskip
\noindent {\bf Question 2.} {\em
Is it true that if $u$ is a solution of \eqref{01}, $|u|<1$, satisfying 
$$ \lim_{x_N\rightarrow -\infty} u(x',x_N) =0\  \mbox{ and } \lim_{x_N\rightarrow \infty} u(x',x_N) =1\ \mbox{uniformly in}\ x'\in\R^{N-1}$$
then $u$ is 1 dimensional?}
\\
This is more similar in nature to the uniformly elliptic case. 
%In case the answer be positive it one would imply that up to translation the solution is the one constructed here.
Nonetheless, classical proofs of these symmetry results rely heavily on the strong maximum principle, or strong comparison principle, and on the sliding method \cite{BN}, which in general don't hold for the truncated Laplacian (see \cite{BGI,BGI2-}). And in particular they are not true here since we construct ordered solutions that touch but don't coincide.

\medskip
Another remark we wish to make is that, even though the solutions we consider are one dimensional, since  they are viscosity solutions, the test functions are not necessarily one dimensional. Hence the proofs in are not of ODE type.

\medskip
Other surprising results concern Liouville type theorems for \eqref{01}. 
Aronson and Weinberger in \cite{AW} and more explicitly, Berestycki, Hamel and Nadirashvili in \cite{BHN} have proved the following Liouville type result.
 
If $v$ is a bounded non negative classical solution of 
$$ \Delta u +u-u^3=0\  \quad \mbox{ in }\ \R^N$$
then either $u\equiv 0$ or $u\equiv 1$.

Once again this result fails if one replaces the Laplacian with the truncated Laplacian.
%and not only with the functions constructed in Proposition \ref{P1} that are one dimensional.
Indeed we prove that there exists infinitely many bounded non negative smooth solutions of
$$\Pmk(D^2u)+f(u)=0,\ \mbox{in}\ \R^N$$
for a general class of nonlinearities that include $f(u)=u-u^3$. %, that are not constant. 
 This is done by constructing infinitely many radial solutions of $\Pmk(D^2u)+f(u)=0$ 
which are positive in $\R^N$ but tend to zero at infinity.

\medskip
Finally in Section \label{ev} we show a different surprising phenomena related to the so called principal eigenvalue of $\Pmk$. This is somehow different in nature but we believe that it sheds some light to these extremal degenerate operators.

%\section{Preliminaries}
\section{One dimensional solutions}\label{oned}
We consider one dimensional viscosity solutions $u$, i.e. $u(x)=v(x_N)$ for  $x=(x_1,\ldots,x_N)$, of 
the problem
\begin{equation}\label{eq1}
\left\{\begin{array}{l}
\Pmk(D^2u)+u-u^3=0\quad\text{in \;$\mathbb R^N$}\\
\;|u|<1.
\end{array}
\right.
\end{equation}
The main result is the following.

\begin{proposition}\label{P1} Concerning problem \eqref{eq1}, the following hold:
\begin{itemize}
	\item[i)] If $u\in USC(\mathbb R^N)$ is a viscosity one dimensional subsolution  then $u\geq0$.
	\item[ii)] The only classical one dimensional solution is $u\equiv 0$.
	\item[iii)] There exist nontrivial viscosity one dimensional solutions, e.g. $u(x)=v(x_N)$, satisfying either
	\begin{equation}\label{Gib2'}
  \lim_{x_N\rightarrow -\infty} v(x_N) =0\  \mbox{ and } \lim_{x_N\rightarrow \infty} v(x_N) =1.
\end{equation}
or
\begin{equation}\label{Gib3}
  \lim_{x_N\rightarrow \pm\infty} v(x_N) =1.
\end{equation}
	\item[iv)] There are no positive viscosity one dimensional supersolutions.
	\item[v)] If $u\geq 0$ is a viscosity one dimensional supersolution  e.g. $u(x)=v(x_N)$ and it is nondecreasing in the $x_N$-direction then there exists $t_0\in\mathbb R$ such that 
	$$u=0\ \mbox{ in }\ X_0=\left\{x\in\mathbb R^N\,:\;x_N\leq t_0\right\}.$$
\end{itemize}
\end{proposition}
\begin{remark} A consequence of  i) and v) is that there are no  viscosity one dimensional solutions increasing in the $x_N$-direction.
\end{remark}
\begin{proof}
\emph{i)} Fix $\hat x=(\hat x_1,\ldots,\hat x_N)\in\mathbb R^N$ and, for $\alpha>0$, let
$$
\max_{\overline B_1(\hat x)}\left[u(x)-\alpha(x_N-\hat x_N)^2\right]=u(x^\alpha)-\alpha(x_N^\alpha-\hat x_N)^2, \quad x^\alpha\in\overline B_1(\hat x).
$$
Then
$$
\alpha(x_N^\alpha-\hat x_N)^2\leq u(x^\alpha)-u(\hat x)\leq 2,
$$
and 
\begin{equation}\label{eq3}
\lim_{\alpha\to\infty}x_N^\alpha=\hat x_N.
\end{equation}
Moreover, using the one dimensional symmetry, for any $x\in\overline B_1(\hat x)$ we have 
\begin{equation*}
\begin{split}
u(x)-\alpha(x_N-\hat x_N)^2&\leq u(x^\alpha)-\alpha(x_N^\alpha-\hat x_N)^2\\
&=u(\hat x_1,\ldots,\hat x_{N-1},x^\alpha_N)-\alpha(x_N^\alpha-\hat x_N)^2.
\end{split}
\end{equation*}
Hence $u(x)-\alpha(x_N-\hat x_N)^2$ has a maximum in $(\hat x_1,\ldots,\hat x_{N-1},x^\alpha_N)\in B_1(\hat x)$ for $\alpha$ large in view of \eqref{eq3}. Then
$$
u^3(\hat x_1,\ldots,\hat x_{N-1},x^\alpha_N)-u(\hat x_1,\ldots,\hat x_{N-1},x^\alpha_N)\leq\Pmk\left({\rm diag} (0,\ldots,0,2\alpha)\right)=0.
$$
We deduce that $u(\hat x_1,\ldots,\hat x_{N-1},x^\alpha_N)\geq0$ for every $\alpha$ big enough. Using semicontinuity and \eqref{eq3} we conclude
$$
u(\hat x)\geq\limsup_{\alpha\to\infty}u(\hat x_1,\ldots,\hat x_{N-1},x^\alpha_N)\geq0.
$$

\bigskip

\noindent
\emph{ii)} By contradiction, let us assume that $u(x)=v(x_N)$ is a classical solution of \eqref{eq1} and that $v(t_0)\neq0$ for some $t_0\in\mathbb R$. By \emph{i)}, $v(t_0)>0$. Let
$$
\delta^-=\inf\left\{t<t_0\,:\;v>0\;\text{in}\;[t,t_0]\right\}\quad\text{and}\quad \delta^+=\sup\left\{t>t_0\,:\;v>0\;\text{in}\;[t_0,t]\right\}.
$$
If $\delta^-=-\infty$ and $\delta^+=+\infty$, then $v>0$ in $\mathbb R$ and, since $v^3-v<0$, we deduce by the equation \eqref{eq1} that $v''(t)<0$ for any $t\in\R$. In particular $v$ is concave in $\mathbb R$, a contradiction to $v>0$. \\
If $\delta^->-\infty$, then $v(t)>0$ for any $t\in(\delta^-,t_0]$ and $v(\delta^-)=0$. 
Moreover there exists $\xi\in(\delta^-,t_0)$ such that $v'(\xi)>0$. Using the equation \eqref{eq1} we deduce that $v''\leq0$ in $[\delta^-,t_0]$, hence  $v'(\delta)\geq v'(\xi)>0$. This implies that for $\varepsilon$ small enough $v(\delta^--\varepsilon)<0$, a contradiction to \emph{i)}.\\
The case $\delta^+<+\infty$ is  analogous.

\bigskip

\noindent
\emph{iii)} Let $u(x)=\tanh\left(\frac{x_N}{\sqrt{2}}\right)$. Then
$$
D^2u={\rm diag}(0,\ldots,0,u^3-u).
$$
If $x_N\geq0$ then $u^3-u\leq0$ and $\Pmk(D^2u)=u^3-u$, while if $x_N<0$ the function $u$ fails to be a solutions  since $\Pmk(D^2u)=0<u^3-u$.\\
Instead we claim that 
$$
\tilde u(x)=
\begin{cases}
\tanh\left(\frac{x_N}{\sqrt{2}}\right) & \text{if $x_N\geq0$}\\
0 & \text{otherwise}
\end{cases}
$$
is a viscosity solution of \eqref{eq1}. This is obvious for $x_N\neq0$. Now take $\hat x=(\hat x_1,\ldots,\hat x_{N-1},0)$. Since there are no test functions $\varphi$ touching  $\tilde u$ by above at $\hat x$, automatically $\tilde u$ is a subsolution. Let us prove the $\tilde u$ is also a supersolution. Let $\varphi\in C^2(\R^N)$ such that $\varphi(\hat x)=\tilde u(\hat x)=0$ and $\varphi\leq u$ in $B_{\delta}(\hat x)$. Our aim is to show that $\lambda_{N-1}\left(D^2\varphi(\hat x)\right)\leq0$, from which the conclusion follows.\\
Let $W_0=\left\{w\in\mathbb R^N\,:\;w_N=0\right\}$ and for any $w\in W_0$ such that $|w|=1$ let 
$$
g_w(t)=\varphi(\hat x+tw) \qquad t\in(-\delta,\delta).
$$
Since $\varphi$ touches $\tilde u$ from below at $\hat x$ and $\tilde u=0$ when $x_N=0$, we deduce that 
$g_w(t)$ has a maximum point at $t=0$. Then 
$$
g_w''(0)=\left\langle D^2\varphi(\hat x)w,w\right\rangle\leq0.
$$
Using the Courant-Fischer formula 
\begin{equation*}
\begin{split}
\lambda_{N-1}\left(D^2\varphi(\hat x)\right)&=\min_{\dim W=N-1}\max_{\;w\in W,\, |w|=1}\left\langle D^2\varphi(\hat x)w,w\right\rangle\\
&\leq \max_{w\in W_0,\, |w|=1}\left\langle D^2\varphi(\hat x)w,w\right\rangle\leq0
\end{split}
\end{equation*}
as we wanted to show.

%\bigskip
\noindent
As above one can check that, for any $c\geq0$, the  one dimensional function
$$
\tilde u(x)=
\begin{cases}
\tanh\left(\frac{x_N-c}{\sqrt{2}}\right) & \text{if $x_N\geq c$}\\
0 & \text{if $|x_N|<c$}\\
-\tanh\left(\frac{x_N+c}{\sqrt{2}}\right) & \text{if $x_N\leq-c$}
\end{cases}
$$
 is non monotone in the $x_N$-direction and is solutions of \eqref{eq1}.
\begin{center}
\includegraphics[scale=0.5]{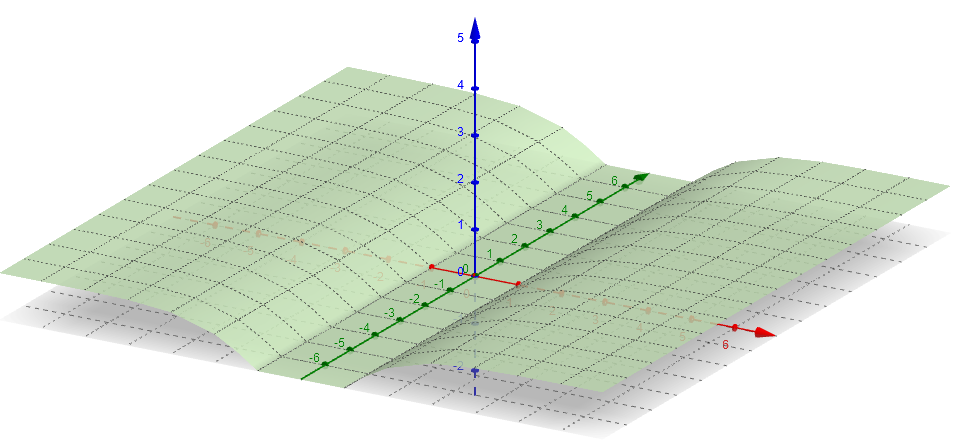}
\\
Fig: The function $\tilde u$
\end{center}

%\bigskip
\medskip

\noindent
\emph{iv)} By contradiction suppose that $u(x)=v(x_N)$ is a positive viscosity supersolution of \eqref{eq1}. We claim that $v$ is strictly concave, leading to a contradiction with $v>0$ in $\mathbb R$. \\
We first prove that $v$ satisfies the inequality
\begin{equation}\label{convex}
v''(t)<0 \quad\text{for any $t\in\R$}
\end{equation} 
in the viscosity sense. For this let $\varphi\in C^2(\R)$ be  test function touching $v$ from below at $t_0$. If we consider $\varphi$ as a function of $N$ variables just by setting  $\tilde\varphi(x)=\varphi(x_N)$, then $\tilde\varphi$ is a test function touching $u$ from below at  $(x_1,\ldots,x_{N-1},t_0)$, for any $(x_1,\ldots,x_{N-1})\in\mathbb R^{N-1}$. Hence
$$
\Pmk\left({\rm diag} (0,\ldots,0,\varphi''(t_0))\right)\leq\varphi^3(t_0)-\varphi(t_0)<0
$$
and then necessarily $\varphi''(t_0)<0$.\\
If $v$ was not strictly concave, then there would exist $t_1<\bar t<t_2\in\R$ such that
$$
\min_{t\in[t_1,t_2]}\left(v(t)-v(t_1)-\frac{v(t_2)-v(t_1)}{t_2-t_1}(t-t_1)\right)=v(\bar t)-v(t_1)-\frac{v(t_2)-v(t_1)}{t_2-t_1}(\bar t-t_1)
$$
Then using  $\varphi(t)=v(t_1)+\frac{v(t_2)-v(t_1)}{t_2-t_1}(t-t_1)$ as a test function in \eqref{convex} we obtain a contradiction.

\bigskip

\noindent
\emph{v)} By  \emph{iv)} there exists $t_0\in\mathbb R$ such that $v(t_0)=0$. By the monotonicity assumption we get $v(t)=0$ for any $t\leq t_0$.
\end{proof}

%\bigskip
\section{Radial solutions}\label{rad}

This section is concerned with the existence of entire radial solutions of the equation

\begin{equation}\label{eq7}
\Pmk(D^2u)+f(u)=0 \quad \text{in $\mathbb R^N$},
\end{equation}
where $f:\mathbb R\mapsto\mathbb R$ satisfies the following assumptions: there exists $\delta>0$ such that 
\begin{equation}\label{assumptionf}
\begin{cases}
f\in C^1((-\delta,\delta))\;\;\text{and it is nondecreasing in $(-\delta,\delta)$}\\
f(u)>0\;\;\;\forall u\in(0,\delta)\\ f(0)=0.
\end{cases}
\end{equation}
Prototypes of such nonlinearities are $$f(u)=\alpha u+\beta |u|^{\gamma-1}u$$ with $\alpha>0$, $\gamma>1$ and any $\beta\in\mathbb R$. 

\begin{proposition}
Under the assumptions \eqref{assumptionf} there exist infinitely many positive and bounded radial (classical) solutions of the equation \eqref{eq7}.
\end{proposition}
\begin{proof}
For any $\alpha\in[0,\delta)$ let $v_\alpha$ be the solution of the initial value problem
\begin{equation}\label{eq8}
\begin{cases}
v_\alpha'(r)+\frac rk f(v_\alpha(r))=0,\;\;r\geq0\\
v_\alpha(0)=\alpha\,
\end{cases}
\end{equation}
defined in its maximal interval $I_\alpha=[0,\rho_\alpha)$.
%Since the function $w_\alpha(t):=v_\alpha(-t)$ is in turn a solution of \eqref{eq8} we infer that $v_\alpha$ is even. Let $I_\alpha=(-\rho_\alpha,\rho_\alpha)$ its maximal interval of existence. 
Since $v_0\equiv0$, then $v_\alpha(r)>0$ for any $r\in I_\alpha$ if $\alpha>0$. We claim that $I_\alpha=[0,\infty)$. For this first note that $v'_\alpha(r)$ is nonpositive in a neighborhood of the origin since, using \eqref{assumptionf}-\eqref{eq8}, one has
$$
v'_\alpha(0)=0\qquad\text{and}\qquad v''_\alpha(0)=-\frac 1kf(\alpha)<0.
$$
If there was $\xi_\alpha\in(0,\rho_\alpha)$ such that $v'_\alpha(r)<0$ in $(0,\xi_\alpha)$ and $v'_\alpha(\xi_\alpha)=0$, then by monotonicity $0<v_\alpha(\xi_\alpha)<\alpha$ and by \eqref{eq8} we should obtain that $f(v_\alpha(\xi))=0$. But this is in  contradiction with \eqref{assumptionf}.
Hence $v_\alpha$ is monotone decreasing and positive, so  $I_\alpha=[0,\infty)$. Using again \eqref{eq8} we deduce moreover  that $\displaystyle\lim_{t\to\infty}v_\alpha(r)=0$ and for any $r>0$
\begin{equation}\label{eq9}
\begin{split}
v''_\alpha(r)&=\frac{v'_\alpha(r)}{r}-\frac rk f'(v_\alpha(r))v'_\alpha(r)\\
&\geq\frac{v'_\alpha(r)}{r}\,.
\end{split}
\end{equation}
In the last inequality we have used the facts that $v_\alpha$ is monotone decreasing, $0<v_\alpha(r)<\alpha$ for any $r>0$ and that $f$ is nondecreasing in $[0,\delta)$ by assumption.\\
By a straightforward computation the function $v_\alpha$ can be written as
\begin{equation}\label{eq172}
v_\alpha(r)=F^{-1}\left(\frac{r^2}{2k}\right),
\end{equation}
$F^{-1}$ being the inverse function of $\displaystyle F(r)=\int_r^\alpha\frac{1}{f(s)}\,ds$ in $(0,\alpha]$.\\
From the above we easily deduce that for any $\alpha\in(0,\delta)$ the radial function
$$
u_\alpha(x)=v_\alpha(|x|)
$$
is a positive radial solution of \eqref{eq7}.
\end{proof}
%
%Let $v(r)$ be the maximal solution of
%\begin{equation}\label{eq4}
%\begin{cases}
%v'(r)=\frac rk(v^3-v), \quad r\geq 0\\
%v(0)=\alpha,\quad \alpha\in(0,1).
%\end{cases}
%\end{equation} 
%Since the constant functions $v_0=0$, $v_{\pm1}=\pm1$ are solutions of the equation, then 
%\begin{equation}\label{eq5}
%0<v(r)<1,\quad \text{$v'(r)<0$\; for $r>0$}.
%\end{equation}
%We deduce that $v$ satisfies \eqref{eq4} in $[0,\infty)$. Moreover
%$$
%v''(r)\geq \frac{v'(r)}{r} \Leftrightarrow \left(\frac{v'}{r}\right)'\geq 0\Leftrightarrow
%(v^3-v)'\geq 0\Leftrightarrow v'(3v^2-1)\geq 0$$
%i.e., since $v'<0$,
%$$v^2\leq\frac13.$$
%
%Since  $v(r)\leq v(0)=\alpha$ for any $r\geq 0$, in view of \eqref{eq5}, if we pick $\alpha\in(0,\frac{\sqrt{3}}{3}]$ then  
%\begin{equation}\label{eq6}
%v''(r)\geq\frac{v'(r)}{r}\;.
%\end{equation}
%Set $$u(x)=v(|x|)\qquad x\in\mathbb R^N.$$
%Using \eqref{eq6} we see that $u$ is a radial solution of \eqref{eq1}:
%$$
%\Pmk(D^2u(x))=k\frac{v'(|x|)}{|x|}=u^3(x)-u(x)\quad \text{for $x\neq0$}
%$$
%and 
%$$
%\Pmk(D^2u(0))=kv''(0)=\alpha^3-\alpha=u^3(0)-u(0).
%$$

In the model case $f(u)=u-u^3$ the assumptions \eqref{assumptionf} are  satisfied with $\delta=\frac{1}{\sqrt{3}}$. By  a straightforward computations, it is easy to see that for any $\alpha\in(0,\frac{1}{\sqrt{3}}]$ the smooth functions
$$
u_\alpha(x)=\frac{1}{\sqrt{1+e^{\frac{|x|^2}{k}+\log\frac{1-\alpha^2}{\alpha^2}}}}
$$
are positive radial  solutions of 
$$
\Pmk(D^2u)+u-u^3=0 \quad \text{in $\mathbb R^N$}.
$$
Let us explicitly remark that the condition $\alpha\in(0,\frac{1}{\sqrt{3}}]$ ensures the validity of the inequality \eqref{eq9}.
\begin{center}
\includegraphics[scale=0.49]{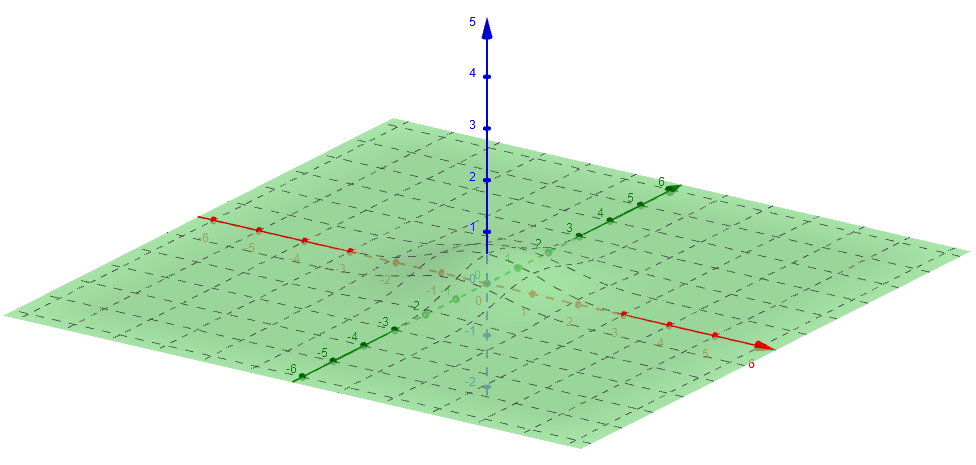}
\\
Fig: The function  
$u_\alpha(x)$
\end{center}

\medskip

\section{Another unusual phenomena}
Let us recall that in \cite{BGI}, given a domain $\Omega$, following Berestycki, Nirenberg and Varadhan \cite{BNV} we define
$$\mu_k^-(\Omega)=\sup\{\mu\in\mathbb R: \exists \phi\in USC(\Omega),\;\phi<0\ \mbox{ s.t. }\ \Pmk(D^2\phi)+\mu\phi\geq 0\ \mbox{ in } \Omega\}.$$
We proved many results among them that $\mu_1^-$ could be called an eigenvalue since under the right conditions on $\Omega$ we construct $\psi<0$ solution of 
$$\Pmo(D^2\psi)+\mu_1^-\psi= 0\ \mbox{in}\ \Omega, \ \psi=0\ \mbox{on}\ \partial\Omega.$$
But, differently from the uniformly elliptic case, $\mu_1^-$  does not satisfy the Faber-Krahn inequality, see \cite{BGI3}.

Another feature of $\mu_k^-$ is that it is the upper bound for the validity of the maximum principle. Precisely if $\mu<\mu_k^-$ and 
$$\Pmk(D^2v)+\mu v \leq 0\ \mbox{in}\ \Omega, \ v\geq 0 \ \mbox{on}\ \partial\Omega$$
then $v\geq 0$ in $\Omega$.

It is well known that for uniformly elliptic operators, the principal eigenvalue goes to 
infinity when the domain decreases to a domain with zero Lebesgue measure.% We had 
%already seen that this in general in not true, but the sequence of domains we had considered were not "narrow". In fact we proved that a sequence of domains collapsing to a circle would have there eigenvalue that stays bounded.

In this section we shall construct a sequence of domains $Q_n\subset\mathbb R^2$ that collapse to a 
segment such that $\mu_1^-(Q_n)$ the principal eigenvalue of $\Pmo$ stays bounded 
above by 1. Hence not only they collapse to a zero measure set, but they are narrower 
and narrower.

Indeed we consider
$$Q_n=\{(x,y)\in\R^2, 0< \frac{nx+y}{2}<\pi,\ \frac{-\pi}{2}<\frac{nx-y}{2}<\frac{\pi}{2}\}.$$
And we define 
$$w_n(x,y)=-(\sin(nx)+\sin y)=-2\sin\left(\frac{nx+y}{2}\right)\cos\left(\frac{nx-y}{2}\right).$$
Obviously $w_n<0$ in $Q_n$ and $w_n=0$ on $\partial\Omega$.
We shall prove that 
\begin{equation}\label{20120}
\Pmo(D^2w_n)+w_n\leq 0 \  \mbox{ in }\ \Omega.
\end{equation}
This will imply that $\mu_1^-(Q_n)\leq 1$. Indeed if $1<\mu_1^-(Q_n)$ the maximum 
principle would imply that $w_n\geq 0$ which is a contradiction. \\
Let us show \eqref{20120}. Clearly
 $$D^2w_n=\left(\begin{array}{cc}
n^2\sin(nx) & 0\\
0 & \sin y
\end{array}\right).
$$

We shall divide $Q_n$ in three areas. In $Q_n\cap(0,\frac{\pi}{n})\times (0,\pi)$ both $\sin(nx)$ and $\sin y$ are positive,
then 
$$\Pmo(D^2w_n)=\min(n^2\sin(nx),\sin y)\leq \sin y\leq \sin(nx)+\sin y=-w_n.$$
In $Q_n\cap (\{-\frac{\pi}{2n}<x\leq 0\}\cup \{\frac{\pi}{n}\leq x<\frac{3\pi}{2n}\})$ where
$\sin(nx)\leq 0$ and $\sin y\geq 0$, then
$$\Pmo(D^2w_n)=n^2\sin(nx)\leq\sin(nx)\leq \sin(nx)+\sin y=-w_n.$$
Finally, in $Q_n\cap (\{\pi\leq y< \frac{3\pi}{2}\}\cup \{-\frac{\pi}{2}<y\leq0\})$ where $\sin(nx)\geq 0$ and $\sin y\leq 0$ we have
$$\Pmo(D^2w_n)=\sin(y)\leq \sin(nx)+\sin y=-w_n.$$
This ends the proof.


\begin{thebibliography}{10}
\bibitem{AW} Aronson, D. G., Weinberger, H. F.: \emph{Multidimensional nonlinear diffusions arising in population genetics}, Adv. Math. 30, (1978), 33-76.
\bibitem{AC} L. Ambrosio, X. Cabr\'e, \emph{Entire solutions of semilinear elliptic equations in $R^3$ and a
conjecture of De Giorgi}, J. Amer. Math. Soc. 13 (2000), 725-739.
\bibitem{BBG} M. T. Barlow, R. F. Bass, C. Gui, \emph{The Liouville property and a conjecture of De Giorgi}, Comm. Pure Appl. Math. 53 (2000), 1007-1038.
%\bibitem{BCN}H. Berestycki, L. Caffarelli, L. Nirenberg, \emph{Monotonicity for elliptic 
%equations in unbounded domains}, Comm. Pure Appl. Math. 50 (1997), 1088-1111.
\bibitem{BHM}H. Berestycki, F. Hamel, R. Monneau, \emph{One-dimensional symmetry 
of bounded entire solutions of some elliptic equations}, Duke Math. J. 103 (2000), 
375-396.
\bibitem{BHN}H. Berestycki, F. Hamel,N. Nadirashvili\emph{The speed of propagation for KPP type problems.
I: Periodic framework}, J. Eur. Math. Soc. 7, (2005) 173-213
\bibitem{BN} H. Berestycki, L. Nirenberg, \emph{On the method of moving planes and the sliding method}, Bol. Soc. Bras. Mat. 22 (1991), 1-37.
\bibitem{BNV} H. Berestycki, L. Nirenberg, S. Varadhan, The principle eigenvalue and maximum principle for second order elliptic operators in general domains, Commun. Pure Appl. Math. 47(1) (1994) 47-92.
\bibitem{BGI} I. Birindelli, G. Galise, H. Ishii, \emph{A family of degenerate elliptic operators: maximum
principle and its consequences}, Ann. Inst. H. Poincar\'e Anal. Non Lin\'eaire 35 (2018), no. 2, 417-441.
\bibitem{BGI2} I. Birindelli, G. Galise, H. Ishii, \emph{Existence through convexity for the truncated Laplacians}, ArXiv:1902.08822v1
\bibitem{BGI3} I. Birindelli, G. Galise, H. Ishii,\emph{Towards a reversed Faber-Krahn inequality for the truncated Laplacian}, to appear on Rev. Mat. Iberoam.
(2020) - DOI 10.4171/RMI/1146
\bibitem{BGI2-} I. Birindelli, G. Galise, H. Ishii,  \emph{Positivity sets of supersolutions of degenerate elliptic equations and the strong maximum principle}, arXiv:1911.08204 
\bibitem{BGL} I. Birindelli, G. Galise, F. Leoni, \emph{Liouville theorems for a family of very degenerate elliptic nonlinear operators}, Nonlinear Analysis, 161 (2017), 198-211.
\bibitem{BL} I. Birindelli, E. Lanconelli,\emph{A negative answer to a one-dimensional symmetry problem in the Heisenberg group.} Calc. Var. Partial Differential Equations {\bf 18} (2003), no. 4, 357-372.
\bibitem{BM}I. Birindelli,  R. Mazzeo, Symmetry for solutions of two-phase semilinear elliptic equations on hyperbolic space. Indiana Univ. Math. J. 58 (2009), no. 5, 2347-2368.
\bibitem{BJ} I. Birindelli, J. Prajapat, \emph{Monotonicity and symmetry results for degenerate elliptic equations on nilpotent Lie groups.} Pacific J. Math. 204 (2002), no. 1, 1-17.
\bibitem{CLN} L. Caffarelli, Y. Y. Li, L. Nirenberg, \emph{Some remarks on singular solutions of nonlinear elliptic equations III: viscosity solutions including parabolic operators},  Comm. Pure Appl. Math. 66 (2013), 109-143.
%\bibitem{CIL}  M. G. Crandall, H. Ishii, P.L. Lions, \emph{User's guide to viscosity solutions of second order partial differential equations}, Bull. Amer. Math. Soc. (N.S.) 27 (1992), no. 1, 1-67.
\bibitem{F} Farina, \emph{Symmetry for solutions of semilinear elliptic equations in $R^N$ and related conjectures}, Ricerche Mat. 48 (1999),  129-154, Papers in memory of Ennio De Giorgi.
%\bibitem{G} G. Galise, \emph{On positive solutions of fully nonlinear degenerate LaneÐEmden type equations}
%Journal of Differential Equations, (2019)
%\bibitem{HL1}F. R. Harvey, H. B. Jr. Lawson, 
%\emph{Dirichlet duality and the nonlinear Dirichlet problem. } {Comm. Pure Appl. Math. }
%62 (2009), no. 3, 396--443. 
%\bibitem{HL1.0}F. R. Harvey, H. B. Jr. Lawson,  \emph{Dirichlet duality and the nonlinear 
%Dirichlet problem on Riemannian manifolds}, J. Differential Geom. 88 (2011), no. 3, 395--482.
\bibitem{P} Polacik, \emph{Propagating terraces in a proof of the Gibbons conjecture and related results},  J. Fixed Point Theory Appl. 19 (2017), no. 1, 113-128.
\end{thebibliography}
\end{document}